\documentclass{birkmult}
 \newtheorem{thm}{Theorem}[section]
 \newtheorem{cor}[thm]{Corollary}
 \newtheorem{lem}[thm]{Lemma}
 
 \theoremstyle{definition}
 
 \theoremstyle{remark}

 \numberwithin{equation}{section}

\newcommand{\N}{\rm{I\!N}}

\newcommand{\eps}{\varepsilon}
\newcommand{\Norm}[1]{\Bigl\|#1\Bigr\|}
\newcommand{\norm}[1]{\|#1\|}
\newcommand{\betr}[1]{| #1  |}
\newcommand{\eing}[1]{_{|{#1}}}
\newcommand{\bgl}{\begin{eqnarray}}
\newcommand{\bglst}{\begin{eqnarray*}}
\newcommand{\egl}{\end{eqnarray}}
\newcommand{\eglst}{\end{eqnarray*}}
\newcommand{\Pel}{Pe\l\-czy\'ns\-ki}
\newcommand{\Ref}[1]{(\ref{#1})}
\newcommand{\id}{{{\mathrm i}\mathrm{d}}}

\begin{document}
\title{Phillips' Lemma for L-embedded Banach spaces}

\author{Hermann Pfitzner}
\address{{ANR-06-BLAN-0015}\smallskip\\
Universit\'e d'Orl\'eans,
BP 6759,
F-45067 Orl\'eans Cedex 2,
France
}
\email{hermann.pfitzner@univ-orleans.fr}
\begin{abstract}
In this note the following version of Phillips' lemma is proved.
The L-projection of an L-embedded space - that is of a Banach space which is complemented in its bidual such
that the norm between the two complementary subspaces is additive - is weak$^*$-weakly sequentially continuous.
\end{abstract}
\subjclass{46B20}
\keywords{\bf Phillips' lemma, L-embedded Banach space, L-projection}

\maketitle\noindent
Phillips' classical lemma \cite{Phillips} refers to a sequence $(\mu_n)$ in $\mbox{ba}(\N)$ (the Banach space of finitely bounded measures
on the subsets of $\N$) and states that if $\mu_n(A)\rightarrow 0$ for all $A\subset\N$ then $\sum_k \betr{\mu_n(\{i\})}\rightarrow0$.
It is routine to interpret this result as the weak$^*$-weak-sequential continuity of the canonical projection from the second dual
of $l^1$ onto $l^1$ because this continuity together with $l^1$'s Schur property gives exactly Phillips' lemma.
(Cf., for example,  \cite[Ch.\ VII]{Die-Seq}.)
Therefore the following theorem generalizes Phillips' lemma
(for the definitions see below):
\begin{thm}
The L-projection of an L-embedded Banach space is weak$^*$-weakly sequentially continuous.
\end{thm}
\noindent The theorem will be proved at the end of the paper.\smallskip\\
The theorem has been known in the two particular cases when the L-embedded space in question is the predual of a von Neumann algebra
or the dual of an M-embedded Banach spece $Y$.
In the first case the result follows from \cite[Th.\ III.1]{Ake67}; in the second case $Y$ has \Pel's property (V) (\cite{GoSaa} or \cite[Th.\ III.3.4]{HWW}) and
has therefore, by \cite[Prop.\ III.3.6]{HWW}, what in \cite[p.\ 73]{Lin-Bochner} or in \cite{Uelger-Phillips} is called the weak Phillips property
whence the result by \cite[Prop.\ III.2.4]{HWW}.\medskip\\
{\em Preliminaries.}
By definition a Banach space $X$ is {\em L-embedded} (or an {\em L-summand in its bi\-dual}) if there is a linear projection $P$ on
its bidual $X^{**}$ with range $X$ such that
$\norm{Px^{**}}+\norm{x^{**}-Px^{**}}=\norm{x^{**}}$ for all
$x^{**}\in X^{**}$. The projection $P$ is called L-projection.
Throughout this note $X$ denotes an L-embedded Banach space with L-projection $P$. We have the decomposition $X^{**}=X\oplus_1 X_{{\rm s}}$
where $X_{{\rm s}}$ denotes the kernel of $P$ that is the range of the projection $Q=\id_{X^{**}}-P$.
We recall that a series $\sum z_j$ in a Banach space $Z$ is called
{\em weakly unconditionally Cauchy} (wuC for short) if
$\sum \betr{z^*(z_j)}$ converges for each $z^*\in Z^*$ or, equivalently,
if there is a number $M$ such that
$\norm{\sum_{j=1}^{n}\alpha_j z_j}\leq M\max_{1\leq j\leq n}\betr{\alpha_j}$
for all $n\in\N$ and all scalars $\alpha_j$.
The presence of a non-trivial wuC-series in a dual Banach space is equivalent to the presence
of an isomorphic copy of $l^{\infty}$.
For general Banach space theory and undefined notation we refer to \cite{JohLin}, \cite{LiTz12}, or \cite{Die-Seq}.
The standard reference for L-embedded spaces is \cite{HWW}; here we mention only that besides the Hardy space $H^1$ the preduals of
von Neumann algebras - hence in particular $L^1(\mu)$-spaces and $l^1$ - are L-embedded.
Note in passing that in general an L-embedded Banach space, contrary to $l^1$, need not be a dual Banach space.

The proof of the theorem consists of two halves. The first one states that the L-projection sends a weak$^*$-convergent sequence to a
relatively weakly sequentially compact set. This has already been proved in \cite{Pfi_Phillips-halb}. The second half asserts the existence
of the 'right' limit and can be deduced from the corollary below which states that the singular part $X_{{\rm s}}$ of the bidual is
weak$^*$-sequentially closed. Note that $X_{{\rm s}}$ is weak$^*$-closed if and only if $X$ is the dual of an M-embedded Banach space \cite[IV.1.9]{HWW}.
The following lemma contains the two main ingredients for the proof of the theorem namely two wuC-series $\sum x_k^*$ and $\sum u_k^*$ by means of which
the theorem above will reduce to Phillips' original lemma.
The first one has already been constructed in \cite{Pfi_Phillips-halb}, the construction of the second one is (somewhat annoyingly)
completely analogous, with the r\^{o}les of $P$ and $Q$ interchanged, cf.\ \Ref{gl-def} and \Ref{gl-defbis}.
(For the proof of the theorem it is not necessary to construct both wuC-series simultanuously
but there is no extra effort in doing so and it might be useful elsewhere.)
\begin{lem}
Let $X$ be L-embedded, let $(x_n)$ be a sequence in $X$ and let $(t_n)$ be a sequence in $X_{{\rm s}}$.
Furthermore, suppose that $x+x_{{\rm s}}$ is a weak$^*$-cluster point of the $x_n$ and that, along the same filter on $\N$, $u+u_{{\rm s}}$ is
a weak$^*$-cluster point of the $t_n$
(with $x, u \in X$, $x_{{\rm s}}, u_{{\rm s}}\in X_{{\rm s}}$).
Let finally $x^*, u^*\in X^*$ be normalized elements.

Then there is a sequence $({n_k})$ in $\N$ and
there are two wuC-series $\sum x_k^*$ and $\sum u_k^*$ in $X^*$ such that
\bgl
t_{n_k}(x_k^*)&=&0\label{Gl1}\quad\mbox{ for all } k\in\N,\\
\lim_{k}\,x_k^*(x_{n_k})&=&x_{{\rm s}}(x^*),\label{Gl2}\\
\lim_{k}\,t_{n_k}(u_k^*)&=&u^*(u),\label{Gl2bis}\\
u_k^*(x_{n_k})&=&0\label{Gl1bis}\quad\mbox{ for all } k\in\N.
\egl
\end{lem}
\begin{proof}
Let $1>\eps>0$ and let $(\eps_j)$ be a sequence of numbers decreasing to zero
such that  $0<\eps_j <1$ and $\prod_{j=1}^{\infty} (1+\eps_j)<1+\eps$.

By induction over $k\in \N_0=\N\cup\{0\}$ we shall construct four sequences
$(x_k^*)_{k\in\N_0}$, $(y_k^*)_{k\in\N_0}$,
$(u_k^*)_{k\in\N_0}$ and $(v_k^*)_{k\in\N_0}$  in $X^*$
(of which the first members $x_0^*$, $y_0^*$ $u_0^*$, and $v_0^*$ are auxiliary  elements used
only for the induction) and an increasing sequence $(n_k)$ of indices such that,
for all (real or complex) scalars $\alpha_j$ and with $\beta=x_{{\rm s}}(x^*)$, $\gamma=u^*(u)$, the following conditions hold for all $k\in \N_0$:
\bgl
x_0^*=0, && \norm{y_0^*}=1,                                              \label{gl0}\\
u_0^*=0, && \norm{v_0^*}=1,                                              \label{gl0bis}\\
\Norm{\alpha_0 y_k^*+\sum_{j=1}^{k}\alpha_j x_j^*}
&\leq&
\Bigl(\prod_{j=1}^{k} (1+\eps_j)\Bigr)\max_{0\leq j\leq k}\betr{\alpha_j},
                                                      \quad\mbox{if } k\geq 1,       \label{gl2}\\
\Norm{\alpha_0 v_k^*+\sum_{j=1}^{k}\alpha_j u_j^*}
&\leq&
\Bigl(\prod_{j=1}^{k} (1+\eps_j)\Bigr)\max_{0\leq j\leq k}\betr{\alpha_j},
                                                      \quad\mbox{if } k\geq 1,       \label{gl2bis}\\
t_{n_k}(x_k^*)&=&0,                             \label{gl4}\\
u_k^*(x_{n_k})&=&0,                             \label{gl4bis}\\
y_k^*(x)=0, &\mbox{and} & x_{{\rm s}}(y_k^*)=\beta, \label{gl1a}\\
u_{{\rm s}}(v_k^*)=0, &\mbox{and} & v_k^*(u)=\gamma, \label{gl1abis}\\
\betr{x_k^*(x_{n_k}) - \beta}&<&\eps_{k},     \quad\mbox{if } k\geq 1,\label{gl5}\\
\betr{t_{n_k}(u_k^*) - \gamma}&<&\eps_{k},     \quad\mbox{if } k\geq 1.\label{gl5bis}
\egl
We set $n_0=1$, $x_0^*=0$, $y_0^*=x^*$, $u_0^*=0$ and $v_0^*=u^*$.\\
For the following it is useful to recall some properties of $P$:
The restriction of $P^*$ to $X^*$ is an isometric isomorphism
from $X^*$ onto $X_{{\rm s}}^{\bot}$ with $(P^*y^*)\eing{X}=y^*$ for all $y^*\in X^*$,
$Q$ is a contractive projection and
$X^{***}=X_{{\rm s}}^{\bot}\oplus_{\infty} X^{\bot}$
(where $X^{\bot}$ is the annihilator of $X$ in $X^{***}$).

For the induction step
suppose now that $x_0^*, \ldots, x_k^*$, $y_0^*, \ldots, y_k^*$, $u_0^*, \ldots, u_k^*$, $v_0^*, \ldots, v_k^*$ and $n_0, \ldots, n_k$ have been constructed
and satisfy conditions \Ref{gl0} -  \Ref{gl5bis}.
Since $x+x_{{\rm s}}$ is a weak$^*$-cluster point of the $x_n$ and $u+u_{{\rm s}}$ is a weak$^*$-cluster point of the $t_n$ along the same filter
there is an index $n_{k+1}$ such that
\bgl
\betr{x_{{\rm s}}(y_k^*) - y_k^*(x_{n_{k+1}}-x)}&<&\eps_{k+1}, \label{gl6a}\\
\betr{t_{n_{k+1}}(v_k^*) - (u+u_{{\rm s}})(v_k^*)}&<&\eps_{k+1}, \label{gl6abis}
\egl
Put
\bglst
E&=&{\rm lin}(\{x^*, x_0^*, \ldots, x_k^*, y_k^*, P^*x_0^*, \ldots, P^*x_k^*, P^*y_k^*,\\
&&\quad\quad
u^*, u_0^*, \ldots, u_k^*, v_k^*, P^*u_0^*, \ldots, P^*u_k^*, P^*v_k^*\})\subset X^{***},\\
F&=&{\rm lin}(\{ x_{n_{k+1}}, t_{n_{k+1}}, x, x_{{\rm s}}, u, u_{{\rm s}} \})\subset X^{**}.
\eglst 
Clearly $Q^*x_j^*,\, Q^*y_k^*, Q^*u_j^*,\, Q^*v_k^*\in E$ for $0\leq j\leq k$.
By the principle of local reflexivity there is an operator
$R:E\rightarrow X^*$ such that
\bgl
\norm{Re^{***}}&\leq&(1+\eps_{k+1})\norm{e^{***}},\label{gl7}\\
f^{**}(Re^{***})&=&e^{***}(f^{**}),\label{gl8}\\
R\eing{E\cap X^*}&=&\id_{E\cap X^*}\label{gl9}
\egl
for all $e^{***}\in E$ and $f^{**}\in F$.

We define
\bgl
x_{k+1}^*=RP^*y_k^*  &\mbox{ and }&  y_{k+1}^*=RQ^*y_k^*,\label{gl-def}\\
u_{k+1}^*=RQ^*v_k^*  &\mbox{ and }&  v_{k+1}^*=RP^*v_k^*.\label{gl-defbis}
\egl
In the following we use the convention $\sum_{j=1}^{0} (\cdots) =0$. Then we have that
\bglst
\alpha_0 y_{k+1}^* +\sum_{j=1}^{k+1}\alpha_j x_j^*
&=& R\Bigl(Q^*(\alpha_0 y_k^* +\sum_{j=1}^{k}\alpha_j x_j^*) + P^*(\alpha_{k+1} 
y_k^* +\sum_{j=1}^{k}\alpha_j x_j^*)\Bigr),\\
\alpha_0 v_{k+1}^* +\sum_{j=1}^{k+1}\alpha_j u_j^*
&=& R\Bigl(P^*(\alpha_0 v_k^* +\sum_{j=1}^{k}\alpha_j u_j^*) + Q^*(\alpha_{k+1} 
v_k^* +\sum_{j=1}^{k}\alpha_j u_j^*)\Bigr).
\eglst
Now (\ref{gl2}) (for $k+1$ instead of $k$) can be seen as follows:
\bgl
\lefteqn{\Norm{\alpha_0 y_{k+1}^* +\sum_{j=1}^{k+1}\alpha_j x_j^*}\leq}\nonumber\\
&\stackrel{\Ref{gl7}}{\leq}&
(1+\eps_{k+1})
\Norm{Q^*(\alpha_0 y_k^* +\sum_{j=1}^{k}\alpha_j x_j^*) +P^*(\alpha_{k+1} y_k^* 
+\sum_{j=1}^{k}\alpha_j x_j^*)}\nonumber\\
&=&
(1+\eps_{k+1})\max\Bigl\{\Norm{Q^*(\alpha_0 y_k^* +\sum_{j=1}^{k}\alpha_j x_j^*)},
               \Norm{P^*(\alpha_{k+1} y_k^* +\sum_{j=1}^{k}\alpha_j 
x_j^*)}\Bigr\}\nonumber\\
&\leq&
(1+\eps_{k+1})\max\Bigl\{\Norm{\alpha_0 y_k^* +\sum_{j=1}^{k}\alpha_j x_j^*},
                    \Norm{\alpha_{k+1} y_k^* +\sum_{j=1}^{k}\alpha_j x_j^*}\Bigr\}   \nonumber\\
&\leq&
\Bigl(\prod_{j=1}^{k+1}(1+\eps_j)\Bigr)\max\{\max_{0\leq j\leq k}\betr{\alpha_j},
                                             \max_{1\leq j\leq k+1}\betr{\alpha_j}\}           \nonumber\\
&=&
\Bigl(\prod_{j=1}^{k+1}(1+\eps_j)\Bigr)\max_{0\leq j\leq k+1}\betr{\alpha_j}   \nonumber
\egl
where the last inequality comes from \Ref{gl0} if $k=0$, and from \Ref{gl2}, if $k\geq 1$.

Likewise, (\ref{gl2bis}) (for $k+1$ instead of $k$) is proved.

The conditions (\ref{gl4}) and (\ref{gl1a}) (for $k+1$ instead of $k$) are easy to
verify because $Pt_{n_{k+1}}=0$, $Qx=0$ and $Qx_{{\rm s}}=x_{{\rm s}}$ thus, by \Ref{gl8}
\bglst
t_{n_{k+1}}(x_{k+1}^*)=Pt_{n_{k+1}}(y_k^*)=0,\\
y_{k+1}^*(x)=Qx(y_k^*)=0\quad\mbox{ and }\quad
x_{{\rm s}}(y_{k+1}^*)=Q^*y_k^*(x_{{\rm s}})=x_{{\rm s}}(y_k^*)=\beta.
\eglst
In a similar way we obtain  (\ref{gl4bis}) and (\ref{gl1abis}) (for $k+1$ instead of $k$) by
$u_{k+1}^*(x_{n_{k+1}})=RQ^*v_k^*(x_{n_{k+1}})=Qx_{n_{k+1}}(v_k^*)=0$, $u_{{\rm s}}(v_{k+1}^*)=Pu_{{\rm s}}(v_k^*)=0$ and
$v_{k+1}^*(u)=v_k^*(u)=\gamma$.

Finally, we have
\bglst
x_{k+1}^*(x_{n_{k+1}}) - \beta = y_k^*(x_{n_{k+1}})- \beta =y_k^*(x_{n_{k+1}}-x)- x_{{\rm s}}(y_k^*)
\eglst
by \Ref{gl1a} whence \Ref{gl5} for $k+1$ by \Ref{gl6a}.
Analogously, we get \Ref{gl5bis} for $k+1$ via \Ref{gl6abis} and $t_{n_{k+1}}(u_{k+1}^*)=t_{n_{k+1}}(v_k^*)$ and $(u+u_{{\rm s}})(v_k^*)=\gamma$
by  \Ref{gl1abis}.\\
This ends the induction and the lemma follows immediately.\end{proof}
\begin{cor}
The complementary space $X_{{\rm s}}$ of an L-embedded Banach space $X$ is weak$^*$-sequentially closed.
\end{cor}
\begin{proof}
Suppose that $(s_n)$ is a sequence in $X_{{\rm s}}$ that weak$^*$-converges to $v+v_{{\rm s}}$.
Let $u^*\in X^*$ be normalized, set $t_n=s_n-v_{{\rm s}}$.
We apply the lemma to $(t_n)$ with $u=v$, $u_{{\rm s}}=0$ and $x_n=u$ and
define a sequence $(\mu_n)$ of finitely additive measures on the subsets of $\N$ by
$\mu_n(A)=(t_n-u)(\sum_{k\in A}u_k^*)$ for all $A\subset \N$ where $\sum_{k\in A}u_k^* \in X^*$ is to be understood in the weak$^*$-topology of $X^*$
and where the $u_k^*$ are given by the lemma.
Then $\mu_n(A)\rightarrow0$ for all $A\subset \N$ and by Phillips' original lemma we get
\bglst \betr{t_{n_k}(u_k^*)}\stackrel{\Ref{Gl1bis}}{=} \betr{(t_{n_k}-u)(u_k^*)}\leq\sum_{j}\betr{(t_{n_k}-u)(u_j^*)} = \sum_{j}\betr{\mu_{n_k}(\{j\})} \rightarrow0.
\eglst
Thus $u^*(u)=0$ by \Ref{Gl2bis} and $u=0$ because $u^*$ was arbitrary in the unit sphere of $X^*$.
Hence $(t_{n_k})$ weak$^*$-converges to 0 which  is enough to see that $(s_n)$ weak$^*$-converges to $v_{{\rm s}}$ in $X_{{\rm s}}$.\end{proof}
\begin{proof}\noindent{\em Proof of the theorem:}
Let $X$ be an L-embedded Banach space with L-projection $P$.
Suppose that the sequence $(x_n^{**})$ is weak$^*$-null and that $x_n^{**}=x_n +t_n$ with $x_n=Px_n^{**}$.
Let $x^*$ be a normalized element of $X$.
The sequence $(x_n)$ is bounded and admits a weak$^*$-cluster point $x+x_{{\rm s}}$.
We use the lemma, this time with the wuC-series $\sum x_k^*$, like in the proof of the corollary and
define a sequence $(\mu_n)$ of finitely additive measures on the subsets of $\N$ by
$\mu_n(A)=x_n^{**}(\sum_{k\in A}x_k^*)$ for all $A\subset \N$.
Then $\mu_n(A)\rightarrow0$ for all $A\subset \N$ and by \Ref{Gl1} and Phillips' original lemma we get
\bglst \betr{x_k^*(x_{n_k})}=\betr{x_{n_k}^{**}(x_k^*)}\leq\sum_{j}\betr{x_{n_k}^{**}(x_j^*)} = \sum_{j}\betr{\mu_{n_k}(\{j\})} \rightarrow0.
\eglst
Thus $x_{{\rm s}}(x^*)=0$ by \Ref{Gl2} and $x_{{\rm s}}=0$ because $x^*$ was arbitrary in the unit sphere of $X^*$.
It follows that each weak$^*$-cluster point of the set consisting of the $x_n$ lies in $X$. Hence this set is relatively weakly sequentially compact
by the theorem of Eberlein-\v{S}mulian. If $x$ is the limit of a weakly convergent sequence $(x_{n_m})$ then $(t_{n_m})$ weak$^*$-converges to $-x$.
Hence $x=0$ by the corollary.
This shows that the sequence $(x_n)$ is weakly null and proves the theorem.\end{proof}


\begin{thebibliography}{10}

\bibitem{Ake67}
Ch.~A. Akemann.
\newblock {The dual space of an operator algebra}.
\newblock {\em Trans. Amer. Math. Soc.}, 126:286--302, 1967.

\bibitem{Die-Seq}
J.~Diestel.
\newblock {\em Sequences and Series in Banach Spaces}.
\newblock Springer, Berlin-Heidelberg-New York, 1984.

\bibitem{GoSaa}
G.~Godefroy and P.~Saab.
\newblock Weakly unconditionally convergent series in {$M$}-ideals.
\newblock {\em Math. Scand.}, 64:307--318, 1989.

\bibitem{HWW}
P.~Harmand, D.~Werner, and W.~Werner.
\newblock {\em {$M$}-ideals in {Banach} {Spaces} and {Banach} {Algebras}}.
\newblock Lecture Notes in Mathematics 1547. Springer, 1993.

\bibitem{JohLin}
W.~B. Johnson and J.~Lindenstrauss.
\newblock {\em Handbook of the Geometry of Banach Spaces, Volumes 1 and 2}.
\newblock North Holland, 2001, 2003.

\bibitem{Lin-Bochner}
P.~K. Lin.
\newblock {\em K{\"{o}}the-{Bochner} function spaces}.
\newblock Birkh{\"{a}}user, Boston, 2004.

\bibitem{LiTz12}
J.~Lindenstrauss and L.~Tzafriri.
\newblock {\em Classical Banach Spaces I and II}.
\newblock Springer, Berlin-Heidelberg-New York, 1977, 1979.

\bibitem{Pfi_Phillips-halb}
H.~Pfitzner.
\newblock L-embedded {Banach} spaces and a weak version of {Phillips} lemma.
\newblock Birkh{\"a}user.
\newblock to appear.

\bibitem{Phillips}
R.~S. Phillips.
\newblock On linear transformations.
\newblock {\em Trans. Amer. Math. Soc.}, 48:516--541, 1940.

\bibitem{Uelger-Phillips}
A.~{\"U}lger.
\newblock The weak {Phillips} property.
\newblock {\em Colloq. Math.}, 87:147--158, 2001.

\end{thebibliography}
\end{document}